\documentclass[12pt]{amsart}
\usepackage{amsmath}
\usepackage{graphicx}
\usepackage{epstopdf}
\usepackage{amssymb,amscd,array}
\usepackage[pagewise]{lineno}
\usepackage{color}
\usepackage{float}

\makeindex \makeatletter
\def\captionof#1#2{{\def\@captype{#1}#2}}
\makeatother

\newcounter{tablegroup}
\newcounter{subtable}[tablegroup]

\newtheorem{thm}{Theorem}[section]
\newtheorem{cor}[thm]{Corollary}
\newtheorem{lem}[thm]{Lemma}
\newtheorem{prop}[thm]{Proposition}

\newtheorem{rem}[thm]{\bf Remark}
\newtheorem{exe}[thm]{\bf Example}
\newtheorem{Que}[thm]{Question}
\numberwithin{equation}{section}

\newcommand{\eps}{\varepsilon}

\begin{document}
\title[Dynamics on Hyperspace of pointwise periodic homeomorphisms]
{Dynamics on Hyperspace of Pointwise Periodic Homeomorphisms}

\author{Issam Naghmouchi}

\address{ Issam Naghmouchi, University of Carthage, Faculty
of Sciences of Bizerte, Department of Mathematics,
Jarzouna, 7021, Tunisia.}
 \email{issam.naghmouchi@fsb.rnu.tn and issam.naghmouchi@fsb.ucar.tn}



\begin{abstract} In this paper, we first prove that the topological entropy of induced map of any distal homeomorphism of a compact metric space  is null. Then we  consider induced map $2^f$ of an arbitrary pointwise periodic homeomorphism $f:X\to X$ of a compact metric space $X$, we show that the set of almost periodic points coincides with the set of uniformly recurrent points, i.e. $AP(2^f)=UR(2^f)$. Furthermore, we prove that inside any infinite $\omega$-limit set $\omega_{2^f}(A)$ there is a unique minimal set and this minimal set is an adding machine. As a consequence, $(2^X,2^f)$ has no Devaney chaotic subsystems. In contrast to these rigidity properties, we obtain some results with chaotic flavor. In fact, we prove the following dichotomy, the hyperspace system $(2^X,2^f)$ is either  equicontinuous or choatic with respect to Li-Yorke chaos and $\omega$-chaos. It is shown that the later case occurs if and only if $R(2^f)\setminus AP(2^f)\neq\emptyset$.  This enables us to provide simple examples of pointwise periodic homeomorphisms with chaotic induced systems.
\end{abstract}
\maketitle

\section{\bf Introduction}

 The dynamics on  hyperspace of a given dynamical system, consists to study the orbit of a collection of points (a closed set) simultaneously instead of a single point. Recently, extensive literature has been developed on the study of dynamics on hyperspaces. Authors mainly investigate which properties of a dynamical system are preserved or reflected in its induced hyperspace system. As pointed out by Kwietniak and Oprocha  in \cite{kop}, the first study in this direction dates back to 1975 and was done by Bauer and Sigmund in \cite{4}. In this paper, we investigate the dynamical behavior of hyperspace systems generated by pointwise periodic homeomorphisms. Periodicity could be interpreted as the most regular type of behavior. One may then ask: how complex can the dynamics on hyperspace become if every point in the original system is periodic?

The answer of this question is known in a few specific cases. For instance, according to \cite{aus} the hyperspace system is always equicontinuous if the phase space is a zero-dimensional compact space (see \cite{liop} for further results on the countable case). For higher dimension,  Kolev and Pérouème \cite{kol} proved that recurrent homeomorphisms of a compact surface with negative Euler characteristic are even periodic and so they are equicontinuous as well as their induced hyperspace maps. However, to the best of our knowledge, no systematic study has been conducted in the general setting of compact metric spaces. In this paper, we aim to investigate this question in the broader context. We will show that for a pointwise periodic homeomorphism of a compact metric space, the hyperspace systems exhibit several rigidity properties, including zero topological entropy. At the same time, the non-equicontinuity in the original system is sufficient to produce a variety of chaotic behaviors in the corresponding hyperspace system.

\medskip

This paper is organized as follow: In this section, we introduce the basic notations, definitions, and preliminary results that will be used in subsequent sections. In Section 2, we establish that the topological entropy $h_{top}(2^f)$ is zero for any distal homeomorphism $f:X\to X$ of a compact metric space $X$. Section 3 is devoted to the study of pointwise periodic closed sets with the weak incompressibility property. In sections 4, 5 and 6, we will consider a pointwise periodic homeomorphism $f:X\to X$ of a compact metric space $X$. In section 4,  we mainly prove that for an almost periodic point $A$ in the hyperspace $2^X$, $\cup_{n\in\mathbb{N}f^n(A)}$ is closed in $X$ and $A$ is uniformly recurrent for $2^f$.  We prove in section 5, that inside any infinite $\omega$-limit set of $2^f$ there is a unique minimal set and this minimal set is an adding machine. Finally in section 6, we establish the following dichotomy result: $(2^X,2^f)$ is either  equicontinuous or choatic with respect Li-Yorke chaos and $\omega$-chaos. The later case occurs if and only if $R(2^f)\setminus AP(2^f)\neq\emptyset$. We finish this section by giving a non-equicontinuous pointwise periodic homeomorphism $F:Y\to Y$ for which we construct a set $D\in 2^Y\setminus R(2^F)$  and a set $C\in R(2^F)\setminus AP(2^F)$.

\medskip

 Let $X$ be a compact metric space, the closure (respectively the interior) of a subset $A$ of $X$ is denoted by $\overline{A}$ (respectively $int_X(A)$). For a subset $A$ of $X$ and $\eps>0$, we denote by $B(A,\eps)=\cup_{x\in A}B(x,\eps)$ where $B(x,\eps)$ is the open ball in $X$ with center $x$ and radius $\eps$.
We denote by $2^X$ the hyperspace of all nonempty closed subsets of $X$. For any two subsets $A$ and $B$ of $X$, we denote by $d(A,B)=inf_{x\in A,y\in B}d(x,y)$ and $d(x,A)=d(\{x\},A)$. The Hausdorff metric $d_H$ on $2^X$ is defined as follows : Let $A,B\in 2^X$
$$d_H(A,B)=max\{sup_{x\in A}d(x,B),sup_{y\in B}d(y,A)\}.$$\\
This defines a distance on $2^X$ (see \cite{il}). With this distance, $2^X$ is a compact metric space (\cite{il}).

A dynamical system is a pair $(X, f )$, where $X$ is a compact metric space and $f : X\to X$ is a continuous map. Let $(X, f )$ be a dynamical system then the induced map $2^f:2^X\to 2^X$ is defined by $2^f(A)=f(A)$ for every $A\in 2^X$. It is well known that $2^f$ is continuous with respect to the Hausdorff metric (see, for example, Nadler 1978 \cite{Nadler1978}). The orbit of a point $x\in X$ is the set $\{f^n(x): \ n\in\mathbb{N}\}$. A subset $A$ of $X$ is said to be $f$-invariant (resp. strongly $f$-invariant) if $f(A)\subset A$ (resp. $f(A)=A$). A periodic cycle is a finite collection of closed sets $A_1,\dots,A_N$ such that $f(A_i)=A_{i \ mod(N)}$ for any $i$. A subset $L$ of $X$ has the \textit{weak incompressibility property} if $F \cap\overline{ f(L\ F)} \neq \emptyset $ whenever $F$ is a proper, nonempty and closed subset of $L$. The $\omega$-limit set of a given point $x\in X$ is defined as follow:
$$\omega_f(x)=\cap_{n\in\mathbb{N}}\overline{\{f^k(x): k\geq n\}}$$

$$=\{y\in X: \exists n_1<n_2<\cdots: \lim_{i \rightarrow +\infty} d(f^{n_i}(x),y)=0\}.$$

An $\omega$-limit set $\omega_{f}(x)$ is always a non-empty, closed, strongly
invariant set, i.e. $f(\omega_{f}(x))= \omega_{f}(x)$ (Lemma 2, Chapter IV in \cite{blo}) and has the weak incompressibility property (see Lemma 3, Chapter IV in \cite{blo}). The weak incompressibility property was observed for the first time by Sarkovskii \cite{sar} for $\omega$-limit sets of interval maps.

A finite sequence $(x_i)_{i=0} ^N$ is said to be an \emph{$\epsilon$-chain} if $d(f(x_i), x_{i+1})< \epsilon$ for all indices $i<N$. A set $A$ is \emph{internally chain transitive} if for any pair of points $a,b \in A$ and any $\epsilon>0$ there exists a finite $\epsilon$-chain $(x_i)_{i=0} ^N$ in $A$ with $x_0=a$, $x_N=b$ and $N\geq 1$. Notice that Meddaugh and Raines \cite{med} established that the family of chain internal transitive closed subsets is closed in $2^X$. Moreover, Barwell \textit{et al} \cite{bar} proved that, for closed subsets, internal chain transitivity is equivalent to weak incompressibility.

 A pair $(x,y)\in X\times X$ is called \textit{proximal} if $\underset{n\to \infty}\liminf\ d(f^n(x),f^n(y))=0$ otherwise it is called \textit{distal}. If $\underset{n\to \infty}\limsup\ d(f^n(x),f^n(y))=0$, then $(x,y)$ is called \textit{asymptotic}. The stable set $W^s(x)$ of a point $x$ is defined by $W^s(x):=\{y\in X:$  $ (x,y)$ is an asymptotic pair $\}$. Notice that for any $x\in X$ and for any minimal set $M\subset \omega_f(x)$, there is $y\in M$ such that $(x,y)$ is proximal (see \cite{onli}).

A point $x$ in $X$ is called
\begin{itemize}
\item \textit{fixed} if $f(x)=x$;
\item \textit{periodic} if $f^n (x)=x$ for some $n\in \mathbb{N}$;
\item \textit{almost periodic}  if for any neighborhood $U$ of $x$ there exists $N\in\mathbb{N}$ such that $\{f^{n+i}(x),\ i=0,1,\cdots,N\} \cap U $ for all $n\in\mathbb{Z}_{+}$,
\item  \textit{recurrent} if $x\in \omega_f(x)$;
\item \textit{uniformly recurrent} if for any neighborhood $U$;
of $x$ there is $N\in\mathbb{N}$ such that $f^{kN}(x)\in U$,  $\forall k\in\mathbb{N}$;
\item  \textit{transitive} if $\omega_f(x)=X$;
\item  \textit{distal} if $(x,y)$ is a distal pair for any $y\in \overline{O_f(x)}\setminus\{x\}$.
\end{itemize}

We denote by $P(f)$, $AP(f)$, $R(f)$  and $UR(f)$ the sets of periodic points, almost periodic points, recurrent points and the uniformly recurrent points of $f$ respectively.  A subset $M$ is called \emph{minimal} if it is non-empty, closed $f$-invariant (\textit{i.e.} $f(A)\subset A$) and there is no proper subset of $M$
having these properties. A subset $M$ of $X$ is called \emph{totally minimal} if it is minimal for every $f^n$.

A dynamical system $(X, f)$ is said to be \textit{(topologically) transitive} for for any non-empty open subsets $U$ and $V$ of $X$ there exists $n\in\mathbb{N}$ such that $f^{-n}(U)\cap V\neq\emptyset$. It is well known that $(X, f)$ is transitive if and only if the set of transitive points is residual in $X$.

 Let $\alpha=(j_1,j_2,\dots)$ be a sequence of integers, where each $j_i\geq 2.$\ Let $\triangle_\alpha$ denote all sequences $(x_1,x_2,\dots)$ where $x_i\in \{0,1\dots,j_i-1\}$ for each $i.$\ We put a metric $d_\alpha$ on $\triangle_\alpha$ given by
$$d_\alpha((x_1,x_2,\dots),(y_1,y_2,\dots))=\displaystyle\sum_{i=1}^{\infty}\frac{\delta(x_i,y_i)}{2^i}$$
 where $\delta$ is the Kronecker symbol. The addition in $\triangle_\alpha$ is defined as follows:
$$(x_1,x_2,\dots)+(y_1,y_2,\dots):=(z_1,z_2,\dots)$$
where $z_1=(x_1+y_1)\ (mod\ j_1$) and $z_2=(x_2+y_2+t_1)\ (mod\ j_2) $, with $t_1=0$ if $x_1+y_1< j_1$ and $t_1=1$ if $x_1+y_1\geq j_1$. So, we carry a one in the second case. Continue adding and carrying in this way for the whole sequence.\\
We define $f_\alpha:\triangle_\alpha\rightarrow \triangle_\alpha$ by
$$f_\alpha(x_1,x_2,\dots)=(x_1,x_2,\dots)+(1,0,0,\dots)$$
We will refer to the system $(\triangle_\alpha, f_\alpha)$ as the \emph{adding machine}.

Let $(X, f )$ and $(Y,g)$ be two dynamical systems. If there is a continuous surjection $\pi: X \to Y$ which intertwines the actions (\textit{i.e.,} $\pi \circ f = g \circ \pi$), then we say that $(Y,g)$ is a factor of $(X, f )$ throughout the factor map $\pi$. If moreover $\pi$ is a homeomorphism then we say that $(X, f )$ and $(Y,g)$ are \textit{(topologically) conjugated}. A closed $f$-invariant subset $A$ of $X$ is called \textit{an adding machine} if $(A,f_{| A})$ is conjugated to an adding machine.

\textit{Variational principle}
Let $(X,f)$ be a dynamical system and lets denote by $h_{top}(f)$ the topological entropy of $f$ and by $h_{\mu}(f)$ the metric entropy of $f$ with respect to an invariant probability measure $\mu$ (see \cite{Walters} for the definitions and for more details). Lets denote by $\mathcal{M}_e(f)$ the set of ergodic invariant measures of $f$. The variational principle runs as follows:
$$h_{top} (f)=Sup\{h_{ \mu}(f): \ \mu\in\mathcal{M}_e(f)\}.$$

\textit{Devaney chaos.} A dynamical system $(X,f)$ is said to be \textit{sensitive} if there is $\delta>0$ such that for any $x\in X$ and for any $\eps>0$, there is $n\in\mathbb{N}$ such that $d(f^n(x),f^n(y))>\delta$ for some $y\in B(x,\delta)$. A \emph{Devaney chaotic} dynamical system $(X, f)$ is any transitive system with dense set of periodic points. For more details on chaos theory see \cite{Li}.

\medskip

\textit{Li-Yorke Chaos.} A pair $(x,y)$ is called a \textit{Li-Yorke pair} if it is proximal but not asymptotic. A dynamical system $(X,f)$ is said to be Li–Yorke chaotic if there exists an uncountable scrambled set of $f$.

\medskip

\textit{Auslander-Yorke chaos.} A dynamical system $(X,f)$ is said to be \textit{Auslander-Yorke chaotic} if it is transitive and sensitive.

\medskip

\textit{$\omega$-Chaos.}
A subset $S$ of $X$ containing at least two points is called an $\omega$-scrambled set for $f$ if for any $x, y \in S$ with
$x\neq y$ the following conditions are fulfilled:
\begin{itemize}
\item[(i)] $\omega_f (x) \setminus \omega_f (y)$ is uncountable,
\item[(ii)]  $\omega_f (x) \cap \omega_f (y) \neq\emptyset$,
\item[(iii)]  $\omega_f (x)$ is not pointwise periodic.
\end{itemize}
A dynamical system $(X,f)$ is said to be \textit{$\omega$-chaotic} if there is an uncountable $\omega$-scrambled set \cite{ll}. Following Theorem 1.4 and
Corollary 2.6 from \cite{NN} the definition of $\omega$-scrambled set is reduced only to conditions (i) and (ii) when
the phase space $X$ is either a completely regular continuum (i.e. every non-degenerate sub-continuum has
non-empty interior) or a zero-dimensional compact space.

\section{\bf Topological entropy of $2^f$}
 In this section we will prove that $h_{top}(2^f)=0$ for any distal homeomorphism $f:X\to X$ of a compact metric space $X$. To achieve this, we will first show that the stable set with respect to $2^f_{|\omega_{2^f}(A)}$ of any recurrent point $A\in 2^X$ reduces to a single point set.This was done using the product recurrence property of distal points: If $x$ is a distal point in a dynamical system $(X,T)$ then for any dynamical system $(Y, S)$ and any recurrent point $y \in Y$,  the pair $(x, y)$ is a recurrent point in the product system $(X \times Y, T \times S)$ (see \cite[Theorem~9.11]{Furst}). Having shown this, we then invoke the variational principle together with the result of Blanchard, Host, and Ruelle (see \cite{blan}), which states that if the metric entropy relative to some ergodic probability measure $\mu$ is positive, then the set of points belonging to a proper asymptotic pair is of measure one with respect to $\mu$.

\medskip
\begin{thm}\label{t1} Let $f:X\to X$ be a distal homeomorphism of a compact metric space $(X,d)$. Then $h_{top}(2^f)=0$.
\end{thm}

We will need the following results to prove Theorem 2.1.

\begin{lem}\label{l1} Let $f:X\to X$ be a distal homeomorphism of a compact metric space $X$ and let $A,B\in 2^X$. If $A\in R(2^f)$ and $(A,B)$ is an asymptotic pair of $2^f$ then  $B\subset A$.
\end{lem}
\begin{proof}
Let $A,B\in 2^X$ be such that $A\in R(2^f)$ and $(A,B)$ is an asymptotic pair of $2^f$. Suppose that $B\setminus A$ is not empty. Take $b\in B\setminus A$ and choose $\delta>0$ so that $B(b,2\delta)\subset X\setminus A$. As $b$ is a distal point, the pair $(b,A)$ is recurrent under the action of the homeomorphism $(f,2^f)$. Thus for a suitable infinite sequence of positive integers $(n_i)_i$,  the limit in the product space $X\times 2^X$,
$$\lim_{i\to +\infty} (f^{n_i}(b),f^{n_i}(A))=(b,A).$$

However as $(A,B)$ is an asymptotic pair for $2^f$, we have also 

$$\lim_{i\to +\infty} (f^{n_i}(b),f^{n_i}(B))=(b,A).$$

It follows that for $i$ large enough, $f^{n_i}(b)\in B(b,\delta)\cap B(A,\delta)$ which is clearly a contradiction. Consequently, $B\setminus A=\emptyset$ and so $B\subset A$.
\end{proof}

\medskip
\begin{prop}\label{p1} Let $f:X\to X$ be a distal homeomorphism of a compact metric space $X$. Then for any $A\in R(2^f)$,
$$W^s(A)\cap \omega_{2^f}(A)=\{A\}.$$
\end{prop}

\begin{proof} Let $A\in R(2^f)$ and suppose that $B\in  W^s(A)$ then $B$ has the same $\omega$-limit set as $A$ under $2^f$. If moreover $B$ belongs to the $\omega_{2^f}(A)$ then it is a recurrent point of $2^f$. By lemma \ref{l1}, $A=B$.
\end{proof}

\begin{proof} [Proof of Theorem \ref{t1}] Let $f:X\to X$ be a distal homeomorphism of a compact metric space $X$. Take $\mu$ an ergodic measure of $2^f$ and denote by $S$ its support. Clearly if $S$ is finite then the metric entropy $h_{\mu}(2^f)=0$. In view of Poincar\'e’s Recurrence Theorem, any isolated point $Z$ of $S$ is periodic and so it has finite orbit with positive $\mu$-measure. As $\mu$ is ergodic, its support $S$ should be reduced to the orbit of $Z$ if $S$ contains an isolated point $Z$ and so it is finite in this case. Therefore, if $S$ is infinite then it has no isolated point. Hence the set of transitive points for the dynamical system $(S,2^f_{\mid S})$ has full measure with respect to $\mu$ (see Theorem 5.16 in \cite{Walters}). Thus for $\mu$-a.e. $A\in 2^X$, $\omega_{2^f}(A)=S$ and $A\in S$. Consequently,  for $\mu$-a.e.
 $A\in 2^X$, $A\in R(2^f)$ and $\omega_{2^f}(A)=S$ thus $W^s(A)\cap S=\{A\}$ by Proposition \ref{p1}. According to Blanchard, Host and Ruette's result in \cite{blan}, $h_{\mu}(2^f)=h_{\mu}(2^f_{|S})=0$. Finally the topological entropy $h_{top}(2^f)$ vanishes by the variational principle Theorem (see Corollary 8.6.1 in \cite{Walters}).
\end{proof}

\begin{rem}\rm{
\begin{itemize}

\item[(i)] The assumption of pointwise distality in Theorem 2.1 is essential. Indeed, there exist well-known examples of minimal Toeplitz subshifts with positive entropy that admit a generic set of distal points. Moreover, this result cannot be strengthened by passing to sequence topological entropy. Since as shown by Qiu and Zhao in \cite{Qiu}, non-equicontinuous distal systems must have infinite sequence topological entropy.

\item[(ii)] In \cite{jw}, we provided an example of a pointwise periodic rational curve homeomorphism $F$ and we claimed that $h_{top}(2^F)=+\infty$ (see Example 4.9 in  \cite{jw}). It is worth noticing that the entropy of $2^F$ was incorrectly computed. In fact according to the above result, it vanishes. However, as we will see later in Section 5, the induced map $2^F$ of this particular example still exhibits chaotic features, such as Li-Yorke chaos.
\end{itemize}

}
\end{rem}

\section{\bf The structure of pointwise periodic invariant closed subset with the weak incompressibility property}
In this section, we will only consider a continuous map $f: X\to X$. By adopting basically the same ideas as in \cite{NN}, we prove that every closed,$f$-invariant and pointwise periodic subset of $X$ that additionally satisfies the weak incompressibility condition has only finitely many connected components, which together form a periodic cycle. This constitutes an extension of Theorem 1.1 in \cite{NN}, which originally addressed only pointwise periodic $\omega$-limit sets. Notice that the class of closed, pointwise periodic, invariant subsets with the weak incompressibility property may be strictly larger than the class of pointwise periodic $\omega$-limit sets. For example, the set $S^1\times \{0\}$ in Example 5.6, satisfies the weak incompressibility property, yet it does not arise as an $\omega$-limit set. As a consequence, if one takes a convergent sequence in $2^X$ of internally chain transitive closed sets so that its limit is pointwise periodic, then this limit necessarily has finitely many connected components which together form a periodic cycle (see Corollary \ref{crucial lemma}).

\begin{thm}\label{t2} Let $f:X\to X$ be continuous map of a compact metric space $X$ and let  $A$ be a closed $f$-invariant subset of $X$ composed uniquely by periodic points. If $A$ has the weak incompressibility property then it has finitely many connected component, say $C_0,\dots,C_{N-1}$ that form a periodic cycle of period $N$.
\end{thm}

\begin{lem}\label{l2} Let $f:X\to X$ be continuous map and $g:Y\to Y$ be a factor of $f$ throughout the factor map $\pi:X\to Y$. Let $A$ be a closed $f$-invariant subset of $X$. If $A$ has the weak incompressibility property then so does $\pi(A)$.
\end{lem}
\begin{proof} According to Theorem 2.2 in \cite{bar}, it suffices to show that $\pi(A)$ is internal chain transitive. Indeed, note that $\pi(A)$ is a closed subset of $Y$. Let $\eps>0$. By uniform continuity of $\pi$, there is $\eta>0$ such that $d(\pi(x),\pi(y))<\eps$ whenever $d(x,y)<\eta$. Recall that $A$ is internal chain transitive thus for any $x,y\in A$ there is an $\eta$-chain for $f$ from $x$ to $y$. So the image of this chain under $\pi$ is clearly an $\eps$-chain for $g$ from $\pi(x)$ to $\pi(y)$.
\end{proof}

\begin{proof} [Proof of Theorem \ref{t2}] Without loss of generality, we may assume that $A=X$ since otherwise we consider $f_{|A}$ instead of $f$. By collapsing all connected components of $X$, we obtain a new quotient space $Y$ and a map $g:Y\to Y$ which is the factor of $f$ throughout the quotient map $\pi:X\to Y$. We shall note first that $P(g)=Y$. By Lemma \ref{l2}, $\pi(X)=Y$ has also the weak incompressibility property for $g$. Recall that $Y$ is totally disconnected and $Y=\cup_{n\in\mathbb{N}}Fix(g^n)$. Thus by Baire's Theorem,  $\cup_{n\in\mathbb{N}}int_{Y}(Fix(g^n))$ is an open dense subset of $Y$. If we assume that $Y$ is infinite then one can choose a small enough non empty clopen subset $O$ of $Y$ included into $Fix(g^N)$ for some $N\in\mathbb{N}$ and such that $O\cup\dots\cup g^{N-1}(O)$ is a proper subset of $Y$. Note that $g^i(O)$ is also a clopen subset of $Y$ for each $i$. Thus $V=O\cup\dots\cup g^{N-1}(O)$ is non-empty proper clopen subset of $Y$ which is further $g$-invariant. This is a contradiction with the weak incompressibility property applied to $V$. It follows that the assumption ¨$Y$ is infinite¨ is impossible. Consequently $\pi(A)=Y$ is a finite set with the incompressibility property so it is exactly a periodic orbit.  This is equivalent to say that $A$ has finitely many connected components that form a periodic cycle.
\end{proof}

\begin{cor}\label{crucial lemma} Let $f:X\to X$ be a continuous map of a compact metric space $(X,d)$ and let $(A_n)_n$ be a sequence of closed sets with the weak incompressibility property that converges in $2^X$ to $C$. If $C\subset P(f)$ then we have the following:
\begin{itemize}
\item[(i)] $C$ has finitely many connected component, say $C_0,\dots,C_{N-1}$, that form a periodic cycle of period $N$.

\item[(ii)] Furthermore, if $A_n=O_f(x_n)\subset P(f)$ for each $n\in\mathbb{N}$ and $\lim_{n\to+\infty}x_n=x\in C_0$ then the limit in $2^X$, $\lim_{n\to +\infty}O_{f^{N}}(x_n)=C_0$.
\end{itemize}
\end{cor}

\begin{proof} (i) Note that $A_n$ is internal chain transitive for any $n$ (see Theorem2.2 in \cite{bar}). According to Lemma 2 in \cite{med}, the limit $C$ is internal chain transitive. Again by Theorem2.2 in \cite{bar}, $C$ has the weak incompressibility property. According to Theorem \ref{t2}, $C$ has finitely many connected component say $C_0,\dots,C_{N-1}$ that form a periodic cycle.

(ii) Denote by $C=C_0\cup\dots,C_{N-1}$. Let $\delta>0$ so that $B(C_i,\delta)$, for $i=0,\dots, N$ are pairwise disjoint. Choose $\eps\in (0,\delta]$ such that $f(B(C_i,\eps))\subset B(C_{i+1(N)},\delta)$ for any $i$. For $n$ large enough, $d_H(A_n,C)<\eps$ and $x_n\in B(x,\eps)$. It follows that for $i=0,\dots, N-1$, $O_{f^N}(f^i(x_n))\subset B(C_i,\delta)$ eventually. On the other hand , for $n$ large enough, $C\subset B(A_n,\eps)$, in particular $C_i\subset B(A_n,\eps)$ for $i=0,\dots, N-1$. Hence for $i=0,\dots, N-1$, $d_H(O_{f^N}(f^i(x_n)), C_i)<\delta$ eventually.
\end{proof}
\begin{rem} \rm{By applying Corollary \ref{crucial lemma}, one can easily show the equicontinuity of any pointwise periodic homeomorphism of a zero-dimensional compact metric space. Consequently, the induced hyperspace system is equicontinuous in this case.}
\end{rem}

\section{\bf The set $AP(2^f)$ of almost periodic points and the set $UR(2^f)$ of uniformly recurrent points}
In this section we establish mainly the following: For any  pointwise periodic homeomorphism $f:X\to X$ of a compact metric space $X$, the set of almost periodic points is equal to the set of uniformly recurrent points for the induced map $2^f$. To accomplish this, we shall first establish several auxiliary results, which are themselves of independent interest

\begin{prop}\label{p41} Let $f:X\to X$ be a pointwise periodic homeomorphism of a compact metric space $X$.  If  $A\in AP(2^f)$ then $\cup_{n\in\mathbb{N}}f^n(A)$ is a closed strongly $f$-invariant subset of $X$.
\end{prop}

\begin{proof} Denote by $C=\cup_{n\in\mathbb{N}}f^n(A)$. Clearly, $f(C)\subset C$. Any point $x\in A$ can be written as follows $x=f(f^{n-1}(x))$ where $n$ is the period of $x$ so $x\in f(C)$. Thus $f(C)=C$. Take a point $y$ in the closure of $C$ in $X$  and let $N$ be its period. Note that
$$\overline{C}=\bigcup_{l=0}^{N-1}(\overline{\cup_{n\in\mathbb{N}} f^{nN}(f^l(A))}).$$
Hence $y\in\overline{\cup_{n\in\mathbb{N}} f^{nN}(f^l(A))}$ for some $l\in\{0,\dots,N-1\}$. Recall that $f^l(A)\in AP(2^{f^N})$. Thus for a given $\eps>0$ there is $i_0\in\mathbb{N}$ such that for any $n\in\mathbb{N}$, $d_H(f^l(A),f^{(n+j)N+l}(A))<\eps$ for some $j\in\{0,\dots,i_0-1\}$. From the continuity of $2^{f^N}$, there is $0<\eta<\eps$ such that $f^{jN}(B(y,\eta))\subset B(y,\eps)$ for any $j\in\{0,\dots,i_0-1\}$. As $y\in\overline{\cup_{n\in\mathbb{N}} f^{nN}(f^l(A))}$, there is $x\in f^l(A)$ and $n\in\mathbb{N}$ such that $f^{nN}(x)\in B(y,\eta)$ and so $d(f^{(n+j)N}(x),y)<\eps$ for $j=0,\dots,i_0-1$. But for some $j\in\{o,\dots,i_0-1\}$, $d_H(f^l(A),f^{(n+j)N+l}(A))<\eps$ which implies that $d(y,f^l(A))<2\eps$. This is true for any $\eps>0$ then $y\in\overline{f^l(A)}=f^l(A)$.
 \end{proof}

\begin{rem}\rm{ It is possible for a pointwise periodic homeomorphism to have $R(2^f)\setminus AP(2^f)\neq\emptyset$, see for instance Theorem 6.2 and Example 6.3.}
\end{rem}

\begin{cor}\label{c4} Let $f:X\to X$ be a pointwise periodic homeomorphism of a compact metric space $X$ and let  $A\in AP(2^f)$. Then for any $B\in 2^X$ of the closure in the hyperspace $2^X$ of the orbit of $A$, $$\cup_{n\in\mathbb{N}}f^n(A)=\cup_{n\in\mathbb{N}}f^n(B).$$
\end{cor}

\begin{proof} Let $B\in \overline{O_{2^f}(A)}$. It is clear from Proposition \ref{p41}, that $B\subset \cup_{n\in\mathbb{N}}f^n(A)$  and so $ \cup_{n\in\mathbb{N}}f^n(B)\subset \cup_{n\in\mathbb{N}}f^n(A)$. The other inclusion comes from the fact that $B$ is itself an almost periodic point of $2^f$ and $A$ is a point of the closure of its orbit in the hyperspace $2^X$.
\end{proof}

\begin{thm}\label{th4} Let $f:X\to X$ be a pointwise periodic homeomorphism of a compact metric space $X$.  Then $AP(2^f)=UR(2^f)$.
\end{thm}

\begin{lem}\label{p42} Under the assumptions of Theorem \ref{th4}, an infinite minimal set of $2^f$ is not totally minimal.
\end{lem}
\begin{proof}
Suppose $M$ is an infinite minimal set of $2^f$ and let $A\in M$. By Corollary \ref{c4}, $Y=\cup_{n\in\mathbb{N}}f^n(A)$ is a closed strongly $f$-invariant subset of $X$. Set $g=f_{|Y}$. As $M$ is infinite, $A\subsetneq Y$ so by Baire's Theorem, there exist $N\in\mathbb{N}$ and a non-empty open subset $O$ of $Y$ such that $O\subset Fix(g^n)\setminus A$. Observe that $Y\setminus O$ is a closed strongly $f^N$-invariant subset of $Y$ and $f^{kN}(A)\cap O=\emptyset$ for every $k\in\mathbb{N}$. Thus for any $B\in\omega_{2^{f^N}}(A)$, $B\in 2^{Y\setminus O}$, in particular $B\cap O=\emptyset$. However, there certainly exists an $n\in\mathbb{N}$ so that $f^n(A)\cap O\neq\emptyset$. Hence $f^n(A)\notin \omega_{2^{f^N}}(A)$.
\end{proof}

\begin{proof} [Proof of Theorem \ref{th4}] Let $A\in AP(2^f)\setminus P(2^f)$ then the closure of $O_{2^f}(A)$ in $2^X$ is an infinite minimal set for $2^f$, lets denote this set by $\mathcal{N}_0$. By Lemma \ref{p42}, $\mathcal{N}_0$ is not totally minimal for $2^f$ thus there is $n_1\in\mathbb{N} $  such that $M_1:=\overline{O_{2^{f^{n_1}}}(A)}$ is a proper subset of $\mathcal{N}_0$ and  so is $\mathcal{N}_1:=\overline{O_{2^{f^{2n_1}}}(A)}$. It follows that
$$\mathcal{N}_0=\mathcal{N}_1\cup\dots\cup f^{2n_1-1}(\mathcal{N}_1).$$

Note also that $\mathcal{N}_1$ is an infinite minimal set for $2^{f^{2n_1}}$. Again by Lemma \ref{p42}, $\mathcal{N}_1$ is not totally minimal for $2^{f^{2n_1}}$ and so there is $n_2\in\mathbb{N}$ such that $M_2:=\overline{O_{2^{f^{2n_1n_2}}}(A)}$ is a proper subset of $\mathcal{N}_1$ and so is $\mathcal{N}_2:=\overline{O_{2^{f^{2n_13n_2}}}(A)}$. Hence
$$\mathcal{N}_1=\mathcal{N}_2\cup f^{2n_1}(\mathcal{N}_2) \dots\cup f^{2n_1(3n_2-1)}(\mathcal{N}_2).$$

And so on we proceed by induction and we get for each $k\in\mathbb{N}$ a minimal set $\mathcal{N}_k:=\overline{O_{2^{f^{m_k}}}(A)}$ for $2^{f^{m_k}}$ such that  $\mathcal{N}_k$ is proper subset of $\mathcal{N}_{k-1}$ and $$m_k=2n_1\times\dots\times (k+1)n_k.$$  Consequently, $(\mathcal{N}_k)_k$ is a nested sequence of compact subsets in the hyperspace $2^X$, thus $\mathcal{N}:=\cap_{k\in\mathbb{N}}\mathcal{N}_k$ is a non-empty compact subset of $2^X$.

We claim that $\mathcal{N}$ is a single point.  Indeed, let $B\in \mathcal{N}$ and let $x\in B$. Suppose that $N\in\mathbb{N}$ is the period of $x$. Thus  $N$ divides $m_N$. Consider $g=f^{m_N}$ then $x\in Fix(g)$ and $A,B\in \mathcal{N}_N$.  Fix $\eps<0$. Recall that $A\in AP(g)$ then there exists $n_0\in\mathbb{N}$ such that for any $n\in\mathbb{N}$, $d_H(g^{n+i}(A),A)<\eps$ for some $i\in\{0,\dots,n_0-1\}$. As $y\in Fix(g)$, there exists $\eta\in (0,\eps]$ such that $g^j(B(x,\eta))\subset B(x,\eps)$ for any $j=0,\dots,i_0$. By Proposition \ref{p41}, $\cup_{n\in\mathbb{N}}g^n(A)$ is a closed strongly $g$-invariant subset of $X$, hence it contains $B$. It turns out that for some $n\in\mathbb{N}$ and some $y\in A$, $g^n(y)\in B(x,\eta)$ which implies that $g^{n+j}(y)\in B(x,\eps)$ for any $j=0,\dots,i_0$. However there is $i\in\{0,\dots,n_0-1\}$ such that $d_H(A,g^{n+i}(A))<\eps$. Thus $d(x,A)<2\eps$ and by letting $\eps\to 0$, we get $x\in A$. In result $B\subset A$.

By Corollary \ref{c4}, $$\cup_{n\in\mathbb{N}}f^n(A)=\cup_{n\in\mathbb{N}}f^n(B).$$
Therefore, if we consider $B$ instead of $A$ then we get by a similar arguments $A\subset B$. In result, $A=B$.

Consequently, $\mathcal{N}$ is a single point and so $\lim_{k\to +\infty}diam (\mathcal{N}_k)=0$. Thus for a given $\eps>0$, there is $k\in\mathbb{N}$ such that $diam(\mathcal{N}_k)<\eps$. Hence $d_H(f^{nm_k}(A),A)<\eps$ for all $n\in\mathbb{N}$.
\end{proof}

\begin{prop}\label{p44} Let $f:X\to X$ be a pointwise periodic homeomorphism of a compact metric space $X$. Then the property of being a closed regularly recurrent set of $2^f$ is stable under finite intersection.
\end{prop}

\begin{proof} Let $A,B\in RR(2^f)$ such that $A\cap B\neq \emptyset$. For a given $\epsilon>0$ let $\eta\in (0,\epsilon]$ be such that $B(A,\eta)\cap B(B,\eta)\subset B(A\cap B,\epsilon)$. Such a number $\eta$ exists since otherwise for any $k\in\mathbb{N}$ there is $x_k\in B(A,\frac{1}{k})\cap B(B,\frac{1}{k})$  lying outside the open subset $B(A\cap B,\epsilon)$ thus the limit of any convergent subsequence of $(x_k)_k$ is a point of $ A\cap B$ and also a point of $X\setminus B(A\cap B,\epsilon)$ which is a contradiction.  For this positive number $\eta$ there is $N\in\mathbb{N}$ such that for all $k\in\mathbb{N}$, $d_H(A,f^{kN}(A))<\eta$ and $d_H(B,f^{kN}(B))<\eta$. So for any $x\in A\cap B$ and for all $k\in\mathbb{N}$, $d(f^{kN}(x), A)<\eta$ and also $d(f^{kN}(x), B)<\eta$ hence $d(f^{kN}(x), A\cap B)<\epsilon$. Consequently $f^{kN}(A\cap B)\subset B(A\cap B,\epsilon)$ for all $k\in\mathbb{N}$.

Now let $x_1,\dots,x_s$ be finite set of points in $A\cap B$ such that $A\cap B\subset B(x_1,\epsilon)\cup\dots\cup B(x_s,\epsilon)$. Let $M\in\mathbb{N}$ be a common period of each $x_i$, i.e $f^M(x_i)=x_i$ for any $i$. Then $A\cap B\subset B(f^{kM}(A\cap B),\epsilon)$ for all $k\in\mathbb{N}$. In conclusion $d_H(A\cap B,f^{kMN}(A\cap B))<\epsilon$ for all $k\in\mathbb{N}$.
\end{proof}


\section{\bf On $\omega$-limit sets of $2^f$}

In this section we will show that any  $\omega$-limit set $\omega_{2^f}(A)$ of the induced map $2^f$ of a pointwise periodic homeomorphism $f:X\to X$ contains a unique minimal set. Furthermore, if this $\omega$-limit set is not minimal then it contains an adding machine minimal set. 

\begin{thm}\label{t51} Let $f:X\to X$ be a pointwise periodic homeomorphism of a compact metric space $X$ and let $A\in 2^X$. Then we have the following assertions:

\begin{itemize}
\item[(i)]  $\omega_{2^f}(A)$ contains a unique minimal set $M$.

\item[(ii)] $\omega_{2^f}(A)$ is finite if and only if $M$ is finite.

\item[(iii)] If moreover $\omega_{2^f}(A)$ is infinite then $M$ is an adding machine.
\end{itemize}
\end{thm}

\begin{prop}\label{55} Let $f:X\to X$ be a pointwise periodic homeomorphism of a compact metric space $X$. The the following are true:

\begin{itemize}
\item[(i)] Let $A\in 2^X$ and $B\in UR(2^f)$. If $A$ and $B$ are proximal under $2^f$ then $A\subset B$.

\item[(ii)] The proximal cell of any periodic point $B\in 2^X$ is reduced to a single point $\{B\}$.
\end{itemize}
\end{prop}

\begin{proof} (i) Let $x\in A$ and suppose that $N$ is the period of $x$. Let $\eps>0$. As $B\in UR(2^{f^N})$ then there is $s\in\mathbb{N}$ such that $d_H(B,f^{snN}(B))<\eps$ for all $n\in\mathbb{N}$. Note further that $A$ and $B$ are proximal under $2^{f^N}$. Thus there is $n\in\mathbb{N}$ such that $d_H(f^{(n+i)N}(A),f^{(n+i)N}(B))<\eps$ for all $i=0,\dots, s-1$. One of the integers $(n+i)N$ is a multiple of $sN$ and so $d(x,B)<2\eps$. Consequently, $x\in \overline{B}=B$.

(ii) Let $B\in P(2^f)$ with period $N$. Let $A\in 2^X$ such that $(A,B)$ is a proximal under $2^f$. As $(A,B)$ is also a proximal pair under $2^{f^N}$ we can assume without loss of generality that $N=1$. By (i), we have already $A\subset B$. Suppose that $B\setminus A$ is not empty then by Baire's Theorem, there is $x\in B$ and $\delta>0$ such that $(B(x,\delta)\cap B)\subset (B\setminus A)\cap Fix(f^l)$ for some $l\in\mathbb{N}$. Note that $(A,B)$ is proximal for $2^{f^l}$ thus $f^{kl}(A)\cap (B(x,\delta)\cap B)\neq\emptyset$ for some $k\in\mathbb{N}$ and so is $A\cap B(x,\delta)\cap B$, a contradiction. Therefore, $A=B$.
\end{proof}

\begin{cor}\label{5.4} Let $f:X\to X$ be a pointwise periodic homeomorphism of a compact metric space $X$. Then the following are true:

\begin{itemize}
\item[(i)] Any transitive subsystem of $(2^X,2^f)$ is without periodic points unless it has a finite phase space which is a periodic orbit. In particular, $(2^X,2^f)$ has no Devaney chaotic subsystems.
\item[(ii)] $(2^X,2^f)$ is topologically transitive if and only if $X$ is trivial (i.e., a single point ). In particular, $(2^X,2^f)$ is neither Auslander-Yorke chaotic nor Devaney chaotic.
\end{itemize}
\end{cor}

\begin{proof} [Proof of Theorem \ref{t51}]  (i) It is well known that for any minimal subset $M$ of $\overline{O_{2^f}(A)}$, $A$ is proximal to a point $B\in M$. Suppose $M$ and $M'$ are two minimal subsets of  $\overline{O_{2^f}(A)}$. Take $(B,B')\in M\times M'$. By Theorem \ref{th4}, $B,B'\in UR(2^f)$ and $A\subset B\cap B'$ by Proposition\ref{55} (i).  Denote by $C=B\cap B'$. By Propsition \ref{p44}, $C\in UR(2^f)$. Observe that $A\subset C\subset B$. Hence if $(n_i)_i$ is an infinite sequence of positve integers such that $\lim_{i\to +\infty}d_H(f^{n_i}(A),f^{n_i}(B))=0$ then so is the limit $\lim_{i\to +\infty}d_H(f^{n_i}(C),f^{n_i}(B))$. It follows that $C$ and $B$ are proximal under $2^f$. According to Proposition \ref{55} (i), $C=B$. This implies that $B\subset B'$. Similarly if we consider $B'$ instead of $B$ we get the other inclusion $B'\subset B$. Finally $B=B'$ and so $M=M'$. .

(ii) Let $B\in M$ such that $(A,B)$ is proximal for $2^f$. If $M$ is finite then it is a periodic orbit and so $B\in P(2^f)$. By Proposition \ref{55} (ii), $A=B$ which implies that $\omega_{2^f}(A)=\omega_{2^f}(B)$ is finite.

(iii) Suppose that $\omega_{2^f}(A)$ is infinite. Then by Proposition \ref{55} and Theorem \ref{th4},  $M$ is an infinite minimal set all of its points are uniformly recurrent. It turns out that  $M$ is an adding machine by \cite{blo2}.
\end{proof}


\section{\bf Chaos for $2^f$}
It is well known that there is no universally accepted definition of chaos. Consequently, various alternative definitions of chaos have been proposed in the literature (see, for example, \cite{Li} for further discussion). This naturally leads us to the following question:

\begin{Que}
What types of chaos can occur in the corresponding hyperspace dynamical system when all points in the base space are periodic? \end{Que}
As shown above, neither positive entropy, Devaney chaos nor Aulander-Yorke chaos can occur in this case. However, as we will see in Theorem \ref{chaos},  other well-known types of chaos do arise.

\begin{thm}\label{chaos} Let $f:X\to X$ be a pointwise periodic homeomorphism of a compact metric space $X$. The following are equivalent:
\begin{itemize}
\item[(1)] $(X,f)$ is not equicontinuous;

\item[(2)] $(2^X,2^f)$ is not equicontinuous;

\item[(3)] There exists an uncountable set $S\subset R(2^f)$ which is both a scrambled and an $\omega$-scrambled set for $2^f$. In particular,  $(2^X,2^f)$ is both  Li-Yorke and $\omega$-chaotic;

\item[(4)] $(2^X,2^f)$ has a Li-Yorke pair;

\item[(5)] $(2^X,2^f)$ has an $\omega$-scrambled set with cardinal at least two;

\item[(6)] $R(2^f)\setminus AP(2^f)\neq\emptyset$.
\end{itemize}
\end{thm}

\begin{proof} The equivalence between (1) and (2) is already known (see Proposition 7, \cite{baur}). Trivially any conditions from (3) to (6) implies (1) and so (2). Further, it is clear that (3) implies both (4) and (5) and (3) implies (6). It remains to prove that (1) implies (3) and (6).  We will begin by proving that $(1)$ implies $(6)$. Assume that $(X,f)$ is not equicontinuous then there is a sequence $(x_k)_k$ of points in $X$ that converges to $x$ and a constant $\delta>0$ such that for any $k\in\mathbb{N}$, there is $n_k\in\mathbb{N}$ for which $d(f^{n_k}(x),f^{n_k}(x_k))>\delta$. Without loss of generality, one can assume that $x\in Fix(f)$ since $x$ is also a non-equicontinuous point of $(X,f^m)$ where $m$ is the period of $x$. By applying Corollary \ref{crucial lemma}, the sequence $(O_f(x_k))_k$ of periodic orbits converges in the hyperspace $2^X$ to a periodic subcontinuum $C$  with period $1$. Thus $C$ is nondegenerate. Indeed, otherwise the space $Z=(\cup_{k\in\mathbb{N}}O_f(x_k))\cup C$ is a compact countable $f$-invariant subset of $X$ which implies that all point (in particular $x$) in $Z$ are equicontinuous with respect to $f_{|Z}$, a contradiction. We will built a closed set $A$ of $X$ such that $A\in R(2^f)\setminus UR(2^f)$.

Denote for each $k\in\mathbb{N}$ by $l_k$ the period of $x_k$. We will construct a subsequence $(y_s)_s$ of $(x_k)_k$ as follow: First let $y_1=x_1$ and let $N_1$ be its period. Observe that $\cap_{j=0}^{N_1}f^{-j}(B(x,1/2))$ is an open neighborhood of $x$ in $X$. There exists $n_2\in\mathbb{N}$ such that $x_{n_2}\in \cap_{j=0}^{N_1}f^{-j}(B(x,1/2))$ so let $y_2=x_{n_2}$ and $N_2$ be its period. Now suppose that we have already defined $y_1,\dots,y_s$ where $N_i$ is the period of each $y_i$. Thus there is $x_{s+1}\in \cap_{j=0}^{N_1\times\dots\times N_s}f^{-j}(B(x,1/(s+1)))$ so let $y_{s+1}=x_{n_{s+1}}$ and let  $N_{s+1}$ be its period. In this way, $\lim_{s\to +\infty}y_s=x$. Let $A=\{y_s, \ s\in\mathbb{N}\}\cup\{x\}$ then clearly $A\in 2^X$. For any $s\in\mathbb{N}$, $$d_H(A, f^{N_1\times\dots\times N_s}(A))<1/s.$$ Therefore $A\in R(2^f)$. Recall that $A$ is countable then so is $\cup_{n\in\mathbb{N}}f^n(A)$. However the closure of $\cup_{n\in\mathbb{N}}f^n(A)$ contains the nondegenerate continuum $C$ hence it is uncountable. Then by Proposition 4.1, $A\notin UR(2^f)$. This completes the proof of (1) implies (6).

Now we prove that (1) implies (3): We will consider the same $C$ as above. Notice that for any $n\in\mathbb{N}$, a scrambled set for $(2^f)^n$ is also a scrambled for $2^f$ thus we may assume again that $f(C)=C$.  By applying Baire's Theorem to $C=\cup_{n\in\mathbb{N}}C\cap Fix(f^n)$, there is $y\in C$ and $\delta>0$ such that $B(y,\delta)\cap C\subset Fix(f^L)$ for some $L\in\mathbb{N}$. Observe that $y$ is the limit of a sequence of points $(u_k)_k$ such that each $u_k$ is a point from the orbit of $x_k$. By applying Corollary \ref{crucial lemma}, the sequence $(O_{(2^f)^L}(u_k))_k$ has at least a convergent subsequence that converges to a nondegenerate subcontinuum $C'$ of $C$ that contains $y$. For this reason and for simplicity, we may again assume that $L=1$ and $C'=C$. In a similar way as above, there is  a subsequence $(y_k)_k$ of $(u_k)_k$ such that the closed subset  $A=\{y_k: \ k\in\mathbb{N}\}\cup\{y\}$ of $X$ is recurrent under the map $2^f$. Recall that $C\subset \overline{\cup_{n\in\mathbb{N}}f^n(A)}$. By Theorem \ref{th4}, there is $B\in UR(2^f)$ such that $(A,B)$ is a proximal pair under $2^f$. As $C$ is a nondegenerate continuum then the open neighborhood $B(y,\delta)\cap C$ of $y$ in $C$ is uncountable and so does $(B(y,\delta)\cap C)\setminus A$ (since $A$ is countable). We associate to each $z\in (B(y,\delta)\cap C)$, the set $A_z=A\cup\{z\}$. Clearly  $S=\{A_{z}: \ z\in (B(y,\delta)\cap C)\setminus A\}$ is a collection of pairwise distinct closed subsets of $X$. As $A\in R(f)$ then $\lim_{i\to +\infty}d_H(f^{n_i}(A),A)=0$ for some infinite sequence $(n_i)_i$ of positive integers. However as $(B(y,\delta)\cap C)\subset Fix(f)$,  for any $i$

$$d_H(f^{n_i}(A_z),A_z)=d_H(f^{n_i}(A),A). \ (*)$$

This implies that $A_z\in R(2^f)$ for any $z\in (B(y,\delta)\cap C)$.

\textit{Claim 1.} $S$ is a scrambled set for $2^f$: Relation (*) above gives that $(A_z,A_t)$ is not asymptotic under $2^f$ for any $z,t\in(B(y,\delta)\cap C)\setminus A$ such that $z\neq t$ (see also  Proposition 2.1 sonce all $A_z$ are recurrent for $2^f$). We shall prove now the proximal relation in $S$. Observe that $$C\cap B(y,\delta)\subset \overline{\cup_{n\in\mathbb{N}}f^n(A)}.$$ Further according to Propositions \ref{55} and \ref{p41},
$$\overline{\cup_{n\in\mathbb{N}}f^n(A)}\subset \cup_{n\in\mathbb{N}}f^n(B).$$
Hence, we get $$C\cap B(y,\delta)\subset B.$$ It follows that for each $z$, $A\subset A_z\subset B$. Recall that $(A,B)$ is a proximal pair under $2^f$ so there is an infinite sequence $(m_i)_i$ of positive integers such that $$\lim_{i\to +\infty} d_H(f^{m_i}(A),f^{m_i}(B))=0.$$ So for any $i$, $$d_H(f^{m_i}(A_z),f^{m_i}(B))\leq d_H(f^{m_i}(A),f^{m_i}(B)).$$ It follows that $\lim_{i\to +\infty} d_H(f^{m_i}(A_z),f^{m_i}(B))=0$. From the triangular inequality, $\lim_{i\to +\infty }d_H(f^{m_i}(A_z),f^{m_i}(A_{z'}))=0$ for any $z,z'\in (B(y,\delta)\cap C)$. This completes the proof of the claim 1.

\textit{Claim 2.} $S$ is an $\omega$-scrambled set for $2^f$: Notice that $B$ lies inside each $\omega_{2^f}(A_z)$ so $$\omega_{2^f}(A_z)\cap \omega_{2^f}(A_{z'})\neq\emptyset, \ \forall z,z'\in (B(y,\delta)\cap C)\setminus A.$$
 Furthermore, take $z,z'\in  (B(y,\delta)\cap C)\setminus A$ such that $z \neq z'$. Obviously,  $z\in T$ for any $T\in \omega_{2^f}(A_z)$. Immediatly, $A_{z'}\notin \omega_{2^f}(A_z)$. Hence, $\omega_{2^f}(A_{z'})\setminus \omega_{2^f}(A_{z})$ is a non-empty open set of $\omega_{2^f}(A_{z'})$ containing all transitive point of the system $(\omega_{2^f}(A_{z'}),2^f_{\omega_{2^f}(A_{z'})})$. So $\omega_{2^f}(A_{z'})\setminus \omega_{2^f}(A_{z})$ being a residual subset of the uncountable compact  perfect set $\omega_{2^f}(A_{z'}) $, is itself uncountable. This completes the proof of Claim 2.

\end{proof}

In the following, we will give a simple pointwise periodic homeomorphism $F$ for which we will construct a recurrent but not uniforly recurrent point under the induced map.

\begin{exe}  \rm{We will define $Y$ as a subset of the space $\mathbb{C}\times\mathbb{R}$ as follow: First lets denote for each $n\in\mathbb{N}$ and $k\in\{0,\dots,n-1\}$, $u_n^k=(exp(2ik\pi/n),1/n)$ and let $U_n=\{u_n^k: \ k=0,\dots, n-1\}$. Denote by $S^1$ the unit circle in the complex plane $\mathbb{C}$. So let
$$X=(\cup_{n\in\mathbb{N}}U_n)\cup( S^1\times \{0\}).$$
Now we define $F$ as follow: For each $n\in\mathbb{N}$ and $k\in\{0,\dots,n-1\}$, let $F(u_n^k)=u_n^{k+1(n)}$ and $F(z,0)=(z,0)$ for any $z\in S^1$. Clearly $Y$ is a compact metric space and $F:Y\to Y$ is a non-equicontinuous pointwise periodic homeomorphism. Thus $(2^Y,2^F)$ is Li-Yorke chaotic.

An example of recurrent but not almost periodic point for $2^F$: We define the set $$C=\{u_{\frac{1}{2^{2^n}}}^1: \ n\in\mathbb{N}\}\cup \{(1,0)\}.$$

Clearly, $C$ is a closed subset of $Y$. Observe that  $\cup_{n\in\mathbb{N}}f^n(C)$ is countable since $C$ is countable. By Proposition 4.1, $C$ is not almost periodic point for $2^f$. We are going now to show that $C\in R(2^f)$: For every $n\in\mathbb{N}$, we have
$$F^{2^{2^n}}(C)=\{(1,\frac{1}{2^{2}}),\dots,(1,\frac{1}{2^{2^n}})\}\cup\{u_{2^{2^k}}^{2^{2^n}}: \ k\geq n+1\}\cup \{(1,0)\}.$$

It follows that $d_H(F^{2^{2^n}}(C),C)=d(u_{2^{2^{n+1}}}^{2^{2^n}}, C)$ which goes to zero as $n\to +\infty$.

An example of a closed set $D$ not recurrent under $2^F$: We define the sets $$D=\{(1,\frac{1}{n}): \ n\in\mathbb{N}\}\cup \{(1,0)\},$$
$$H=\{(1,\frac{1}{n}): \ n\in\mathbb{N}\}\cup( S^1\times \{0\}).$$
For $\eps>0$, let $N\in\mathbb{N}$ be such that $\frac{1}{N}<\eps$. Then  $d_H(F^{kN!}(H),H)<\eps$ for all $k\in\mathbb{N}$. It follows that $H\in UR(2^F)$.
Observe also that $\lim_{n\to +\infty}d_H(F^n(D),F^n(H))=0$. Therefore $D\notin R(2^F)$.

}
\end{exe}

\textit{Acknowledgements.} This work was supported by the research unit:
``Dynamical systems and their applications'' (UR17ES21), Ministry of Higher Education and Scientific Research,
Tunisia.

\bibliographystyle{amsplain}

\begin{thebibliography}{9}
\bibitem{aus} J. Auslander, E. Glasner and B. Weiss, \emph{On recurrence in zero dimensional flows,} Forum Math-
ematicum 19(2007), 107–114.
\bibitem{bar} A. Barwell, C. Good, P. Oprocha and B.E. Raines, \emph{Characterizations of $\omega$-limit sets in topologically hyperbolic systems}, Discrete. Contin. Dyn. Syst., 33(5):1819-1833.
\bibitem{baur} W. Bauer and K. Sigmund, \emph{Topological dynamics of transformations induced on the space of
probability measures}, Monatshefte fur Mathematik 79 (1975), 81–92.
\bibitem{onli} F. Blanchard, E. Glasner, S. Kolyada and A. Maas, \emph{On Li-Yorke pairs. J. Reine Angew.} Math. 547 (2002),
51–68.
\bibitem{blan} F. Blanchard, B. Host, and S. Ruette, \emph{Asymptotic pairs in positive-entropy systems}, Ergodic
Theory Dynam. Systems 22 (2002), no. 3, 671–686.
\bibitem{blo} L.S. Block, W.A. Coppel, Dynamics in One Dimension, Lecture Notes in Math., vol. 1513,  Springer-Verlag, 1992.
\bibitem{blo2} L. Block and J. Keesling, \emph{A characterization of adding machine maps}, Topol . Appl.
140 (2004) 151–161.

\bibitem{jw} A. Daghar and I. Naghmouchi, \emph{Entropy of induced maps of regular curves homeomorphisms}, Chaos Soliton Fract. 157, 111988 (2022).
\bibitem{Furst} H. Furstenberg, Recurrence in Ergodic Theory and Combinatorial Number Theory, M. B. Porter
Lectures, Princeton University Press, Princeton, N.J., 1981.
\bibitem{kol} B. Kolev and M.-C. Pérouème, \emph{Recurrent surface homeomorphisms,} Math. Proc. Cambridge Philos. Soc.
124 (1998), 161–168.
\bibitem{kop}D. Kwietniak and P. Oprocha, \emph{Topological entropy and chaos for maps induced on
hyperspaces,} Chaos Solitons Fractals 33 (2007), no. 1, 76–86.
\bibitem{Kur} K. Kuratowski, Topology, vol. II, Academic Press, New York and London, 1968.
\bibitem{liop} J. Li, P. Oprocha, X. Ye, and R. Zhang, \emph{When  all closed subsets are recurrent?},Ergodic
Theory Dynam. Systems, 37 (2017), 2223–2254.
\bibitem{Li} J. Li and X. Ye, \emph{Recent Development of Chaos Theory in Topological Dynamic},.
 \bibitem{ll} S. Li, \emph{$\omega$-chaos and topological entropy,} Trans. Am. Math. Soc. 339 (1993) 243–249.
\bibitem{NN} H. Marzougui, I. Naghmouchi, \emph{On Totally Periodic $\omega$-Limit Sets}, Houston Journal of Mathematics, volume 43, No. 4, 2017.
\bibitem{med} J. Meddaugh and B.E. Raines, \emph{Shadowing and internal chain transitivity}, Fund. Math., 222(3):279-287, 2013.
\bibitem{Nadler} S. B. Nadler, Continuum Theory: An Introduction, (Monographs and Textbooks in Pure and Applied Mathematics, 158). Marcel Dekker, Inc., New York, 1992.
\bibitem{Nadler1978} R.E.Nadler, Jr., Hyperspaces of Sets, Marcel Dekker, New York, 1978.
\bibitem{il} Illanes A, Nadler Jr SB. Hyperspaces. Monographs and textbooks in pure and applied mathematics, vol. 216. New York: Marcel
Dekker Inc.; 1999.

\bibitem{Qiu} J. Qiu and J. Zhao, \emph{Null systems in the non-minimal case.} Ergodic Theory and Dynamical Systems, 40 (2020), 1-18.
\bibitem{sar} A. N. Sarkovskii, \emph{Continuous mapping on the limit points of an iteration sequence.} Ukrain.
Mat. Z., 18(5):127–130, 1966.
\bibitem{4} Walter Bauer and Karl Sigmund, \emph{Topological dynamics of transformations induced on the
space of probability measures,} Monatsh. Math. 79 (1975), 81–92.
\bibitem{Walters} P. Walters. An Introduction to Ergodic Theory. Springer, Berlin, 1982.
\end{thebibliography}

\end{document}